\documentclass[12pt,reqno]{amsart} 
\usepackage[margin=1.5in]{geometry}

\usepackage{amsmath,amsthm,amssymb}
\usepackage{physics}
\usepackage{mathrsfs}
\usepackage[shortlabels]{enumitem}
\usepackage{stmaryrd}
\usepackage{shuffle}
\usepackage{multicol} %multiple columns
\usepackage{mathtools}
\usepackage{commath}%for \abs
\usepackage{ytableau}
\usepackage{caption}
\usepackage{subcaption}
\usepackage{hyperref}
\usepackage{cleveref}
\usepackage{subcaption}

\newcommand{\R}{\mathbb{R}}
\newcommand{\Z}{\mathbb{Z}}

\renewcommand{\S}{\mathfrak{S}}

\DeclareMathOperator{\conv}{conv}
\DeclareMathOperator{\PF}{PF}

\DeclareMathOperator{\NC}{NC}

\renewcommand{\phi}{\varphi}

\newtheorem{theorem}{Theorem}[section]

\newtheorem{prop}[theorem]{Proposition}
\newtheorem{lemma}[theorem]{Lemma}

\theoremstyle{definition}

\newtheorem{example}[theorem]{Example}
\usepackage{pdfpages}
\usepackage{comment}
\usepackage{mathtools}

\definecolor{wakegold}{RGB}{158, 126, 56}
\definecolor{purple_opposite}{RGB}{97, 129, 199}

\usepackage{tikz}
\usetikzlibrary{calc}

\title{Exploring Parking Functions: Poset and Polytope Perspectives}
\author{Yan Liu}

\begin{document}
\maketitle

\begin{abstract}
This paper provides an exploration of parking functions, a classical combinatorial object. We present two viewpoints on their structure and properties: through poset of noncrossing partitions and polytopes.
\end{abstract}

\section{Parking Functions}\label{section 1}

\subsection{Parking functions} Imagine living on a one-way street that dead-ends with $n$ parking spots available. You and your neighbors have $n$ cars in total, and everyone has their preferred spot to park. Without reversing, does there exist a solution that everyone can park without collision? In mathematics, this real life dilemma is called the parking problem. Let's formalize this problem.
\begin{itemize}
    \item[a.] There are $n$ parking spots labeled as $1,2,\dots,n$ on a one-way street. There are $n$ cars that want to park here, where $n\in\Z^+$;
    \item[b.] Car $C_i$ is the $i$-th car to park, which has preferred parking spot $\alpha_i\in[n]=\{1,2,\dots,n\}$. More than one car can have the same preference.
    \item[c.] If the preferred spot of some car had already been occupied, the car will move forward and park in the next available spot. 
    \item[d.] No backward movement is allowed.
\end{itemize}
It is not guaranteed that every car will be able to park before driving to the dead-end. If all $n$ cars are able to park under above conditions, then we say the preference list $\alpha=(\alpha_1,\alpha_2,\dots,\alpha_n)$ is a \textit{parking function}. We use $\PF_n$ to denote the set of all parking functions of length $n$. For simplicity, sometimes we drop the parentheses and commas in the preference list. For example, all parking functions of length $3$ are:
\[\begin{matrix}
    111, & 112, & 121, & 211, & 113, & 131, & 311, & 122,\\
    212, & 221, & 123, & 132, & 213, & 231, & 312, & 321.
\end{matrix}\]

It is very natural to ask: how many parking functions are there? In 1966, Konheim and Weiss counted the number of parking functions.

\begin{theorem}[\cite{AnOccupancyKonheimWeiss}]\label{cardinality of pf}
    For a positive integer $n$, $|\PF_n|=(n+1)^{n-1}$.
\end{theorem}
\begin{proof}
    We add an additional spot $n+1$, and arrange the spots in a circle clockwise. We also allow $n+1$ to be chosen as a preferred spot. With this circular arrangement, all cars are be able to park because they can circle around until an available spot is found. Say the preference list in this set up is a \textit{circular parking function} of length $n$, and we use $\mathrm{CPF}_n$ to denote the set of all circular parking functions of length $n$. Since each of the $n$ cars has $n+1$ choices for its preferred spot, we have
    \[|\mathrm{CPF}_n|=(n+1)^n.\]
    Observe that any circular parking function leaves one spot empty, say the label of the empty spot is $k\in[n+1]$. We relabel the $n+1$ spots in the circle by assigning the label $0$ to the empty spot, the label $1$ to the next spot in the clockwise direction, and so on. Given a circular parking function $\gamma=(\gamma_1,\dots,\gamma_n)$, we can ``unwrap" it to a parking function $\alpha=(\alpha_1,\dots,\alpha_n)$ by shifting each preference down by $k$, i.e., we set $\alpha_i=\gamma_i-k\pmod{n+1}$ for each $i\in[n]$.

    \begin{figure}[h]
    \captionsetup{justification=centerlast} % Center the caption
    \centering
    
    % First subfigure with the first circle
    \begin{subfigure}[b]{0.45\linewidth}
        \centering
        \begin{tikzpicture}[scale=0.75]
            % Define the radius of the circle
            \def\radius{1.5}
            \def\sepLength{0.2} % Length of the separating line segments
            
            % First circle
            \begin{scope}[shift={(0, 0)}]
                \draw (0,0) circle(\radius);
                
                % Define the positions and labels for the first circle, and draw short separating line segments
                \foreach \i/\j/\k in {1/0/1, 2/288/5, 3/216/4, 4/144/3, 5/72/2} {
                    % Place numbers 1, 2, 3, 4, 5 inside the circle at each position (clockwise)
                    \node[font=\large] at (\j:\radius - 0.5) {\i};
                    
                    % Draw short line segments to separate each spot
                    \draw[thick] (\j - 36:\radius) -- (\j - 36:\radius + \sepLength);
                    
                    % Label outside of each spot as C_3, C_2, C_1, C_4
                    \ifnum\i=1 \node[font=\large] at (\j:\radius + 0.5) {$C_3$}; \fi
                    \ifnum\i=5 \node[font=\large] at (\j:\radius + 0.5) {$C_2$}; \fi
                    \ifnum\i=4 \node[font=\large] at (\j:\radius + 0.5) {$C_1$}; \fi
                    \ifnum\i=3 \node[font=\large] at (\j:\radius + 0.5) {$C_4$}; \fi
                }
            \end{scope}
        \end{tikzpicture}
        \caption{$\gamma=(4,4,1,3)$}
        \label{fig:alpha}
    \end{subfigure}
    \hfill
    % Second subfigure with the second circle
    \begin{subfigure}[b]{0.45\linewidth}
        \centering
        \begin{tikzpicture}[scale=0.75]
            % Define the radius of the circle and color
            \def\radius{1.5}
            \def\sepLength{0.2} % Length of the separating line segments
            
            % Second circle
            \begin{scope}[shift={(0, 0)}]
                \draw (0,0) circle(\radius);
                
                % Define the positions and labels for the second circle, and draw short separating line segments
                \foreach \i/\j/\k in {4/0/1, 0/288/5, 1/216/4, 2/144/3, 3/72/2} {
                    % Place numbers 4, 0, 1, 2, 3 inside the circle at each position (clockwise), using wakegold color
                    \node[font=\large, text=wakegold] at (\j:\radius - 0.5) {\i};
                    
                    % Draw short line segments to separate each spot
                    \draw[thick] (\j - 36:\radius) -- (\j - 36:\radius + \sepLength);
                    
                    % Label outside of each spot as C_3, C_2, C_1, C_4
                    \ifnum\i=4 \node[font=\large] at (\j:\radius + 0.5) {$C_3$}; \fi
                    \ifnum\i=3 \node[font=\large] at (\j:\radius + 0.5) {$C_2$}; \fi
                    \ifnum\i=2 \node[font=\large] at (\j:\radius + 0.5) {$C_1$}; \fi
                    \ifnum\i=1 \node[font=\large] at (\j:\radius + 0.5) {$C_4$}; \fi
                }
            \end{scope}
        \end{tikzpicture}
        \caption{$\alpha=(2,2,4,1)$}
        \label{fig 1}
    \end{subfigure}

    \caption{Circular parking function $\gamma$ to parking function $\alpha$.}
    \end{figure}
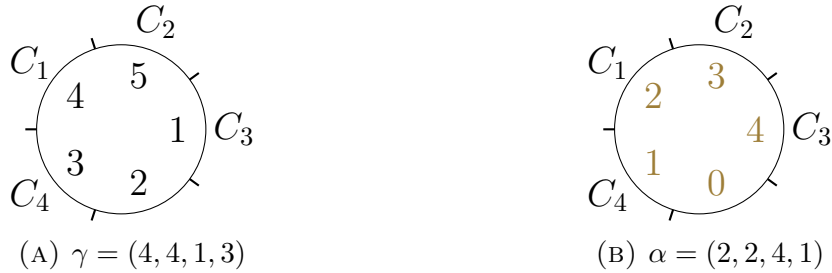

    We say two circular parking functions $\gamma_1,\gamma_2\in\text{CPF}_n$ are \textit{equivalent} if they are related by a circular rotation, i.e.,
    \[\gamma_1\equiv\gamma_2-k\pmod{n+1}\]
    for $k\in[n+1]$. This relation defines an equivalence class, and each equivalence class contains $n+1$ circular parking functions, one for each possible rotation. Since all circular parking functions in the same equivalence class unwrap to the same parking function, we conclude
    \[|\PF_n|=\frac{|\text{CPF}_n|}{n+1}=\frac{(n+1)^n}{n+1}=(n+1)^{n-1}.\qedhere\]
\end{proof}
This elegant idea of adding an extra parking spot and arranging the spots into a circle was by Pollak (see \cite{riordan1969ballots}). 

Parking functions have many nice properties. For example, observe that among the parking functions of length $3$ listed above, some are rearrangements of each other. In general, we have below lemma.
\begin{lemma}\label{parking function rearrangement}
    Every permutation of the entries of a parking function is also a parking function.
\end{lemma}
\begin{proof}
    Suppose $\alpha=(\alpha_1,\dots,\alpha_n)$ is a parking function of length $n$. For any permutation $\sigma\in\S_n$, we aim to show that $\sigma(\alpha)$
    is also a parking function. Since the symmetric group $\S_n$ is generated by adjacent transpositions, it suffices to show that
    \[\alpha'=\sigma(\alpha)=(\alpha_1,\dots,\alpha_{i+1},\alpha_{i},\dots,\alpha_n)\]
    is also a parking function, where $\sigma=(i~i+1)$, $i\in[n-1]$.

    Let $C_i$ and $C_{i+1}$ denote the cars whose preferences are $\alpha_i$ and $\alpha_{i+1}$, respectively, in the preference list $\alpha$. The new preference list $\alpha'$ modifies the order in which cars park. $C_i$ becomes the $(i+1)$-th car and desires $\alpha_{i+1}'$ in $\alpha'$ (which is $\alpha_i$ in $\alpha$). To avoid ambiguity, we denote $C_i$ as $C_{i+1}'$ under the list $\alpha'$. Similarly, $C_{i+1}$ becomes the $i$-th car in $\alpha'$, desiring spot $\alpha_i'$ in $\alpha'$ (which is $\alpha_{i+1}$ in $\alpha$), and is denoted as $C_i'$ under $\alpha'$.

    Assume that under the original preference list $\alpha$, cars $C_i$ and $C_{i+1}$ park at spots $j$ and $k$, respectively, with $j,k\in[n]$. Consequently, we have $\alpha_i\le j$ and $\alpha_{i+1}\le k$. We consider three cases: $\alpha_i=\alpha_{i+1}$, $\alpha_i<\alpha_{i+1}$, and $\alpha_i>\alpha_{i+1}$. In the first case, we have $\alpha=\alpha'$, so $\alpha'$ is a parking function trivially. In the rest of this proof, we will focus on the second case, where $\alpha_i<\alpha_{i+1}$, and show that $\alpha'$ is a parking function. The third case follows by a similar argument and is left to the reader as an exercise.

    If $\alpha_i<\alpha_{i+1}$, if follows that $j<k$. When $C_i'$ is about to park, spot $j$ and $k$ are still available. If $j<\alpha_{i+1}$, $C_i'$ initially attempts to park at its preferred spot $\alpha_{i}'=\alpha_{i+1}$, which exceeds $j$. The next available spot after $\alpha_{i}'$ is $k$ (which may coincide with $\alpha_i'$), and $C_i'$ parks at $k$. Next, when $C_{i+1}'$ is ready to park, it prefers the spot $\alpha_{i+1}'=\alpha_i$. Since $j$ is still available, $C_{i+1}'$ parks at $j$, which is either its preferred spot or the next available spot after $\alpha_{i+1}'$. See \Cref{permutation of a pf is also a pf}. Similarly, if $\alpha_{i+1}\le j$, then $C'_{i}$ parks at $j$ and $C'_{i+1}$ parks at $k$.

    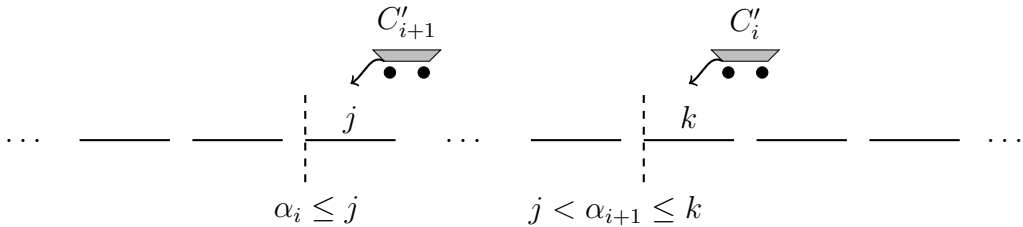
\begin{figure}[h]
    \centering
    \begin{tikzpicture}[scale=1.5]
        % Draw unit intervals as underscores
        \foreach \x in {1,...,3} {
            \draw[thick] (\x, 0) -- (\x+0.8, 0); % Horizontal lines (unit intervals)
        }
            \node at (4.4, 0) {\(\dots\)};
        \foreach \x in {5,...,8} {
            \draw[thick] (\x, 0) -- (\x+0.8, 0); % Horizontal lines (unit intervals)
        }
    
        % Add labels for j and k
        \draw[dashed, thick] (3, 0.4) -- (3, -0.4);
        \node[below] at (3.4, 0.4) {\(j\)};
        \draw[dashed, thick] (6, 0.4) -- (6, -0.4);
        \node[below] at (6.4, 0.4) {\(k\)};
    
        % Add dots before and after
        \node at (0.5, 0) {\(\dots\)};
        \node at (9.2, 0) {\(\dots\)};
    
        % Label for parking spots
        \node[below] at (3.1, -0.4) {\(\alpha_i \leq j\)};
        \node[below] at (5.75, -0.4) {\(j<\alpha_{i+1} \leq k\)};
    
        % Add cartoon car labeled C_{i+1}'
        \draw[fill=gray!50] (3.7, 0.7) -- (4.1, 0.7) -- (4.2, 0.8) -- (3.6, 0.8) -- cycle; % Body of the car
        \draw[fill=black] (3.75, 0.6) circle (0.05); % Front wheel
        \draw[fill=black] (4.05, 0.6) circle (0.05); % Back wheel
        \node[above] at (3.9, 0.8) {\(C_{i+1}'\)}; % Label the car
        % Add arrow indicating that C_i parks at j
        \draw[->, thick] (3.7, 0.7) to[out=150, in=30] (3.4, 0.5);
    
        % Add cartoon car labeled C_i'
        \draw[fill=gray!50] (6.7, 0.7) -- (7.1, 0.7) -- (7.2, 0.8) -- (6.6, 0.8) -- cycle; % Body of the car
        \draw[fill=black] (6.75, 0.6) circle (0.05); % Front wheel
        \draw[fill=black] (7.05, 0.6) circle (0.05); % Back wheel
        \node[above] at (6.9, 0.8) {\(C_i'\)}; % Label the car
        % Add arrow indicating that C_i parks at j
        \draw[->, thick] (6.7, 0.7) to[out=150, in=30] (6.4, 0.5);
    \end{tikzpicture}
    \caption{$C_{i+1}'$ parks at $j$ and $C_{i}'$ at $k$ under $\alpha'$, if $\alpha_i<\alpha_{i+1}$.}
    \label{permutation of a pf is also a pf}
    \end{figure}
    
    Under the preference list $\alpha'$, the remaining $n-2$ cars park in the same spots as in $\alpha$. Since both $C_i'$ and $C_{i+1}'$ successfully park in all three cases, we conclude that $\alpha'$ is a parking function.
\end{proof}

The following proposition provides a simple criterion for determining whether a finite sequence is a parking function; and some people adopt it as the definition of a parking function.

\begin{prop}\label{parking function criterion}
    Let $\alpha=(\alpha_1,\dots,\alpha_n)$ be a finite sequence of positive integers. Let $\beta_1\le\beta_2\le \cdots \le\beta_n$ be the weakly increasing rearrangement of $\alpha$. Then $\alpha$ is a parking function if and only if $\beta_i\le i$ for all $i\in[n]$. 
\end{prop}
\begin{proof}
    Suppose that $\beta_j>j$ for some $j\in[n]$, and the leftmost occurrence of $\beta_j$ in $\alpha$ is at position $i$, i.e., $\alpha_i=\beta_j$, and car $C_i$ desires $\beta_j$. Since $\beta_j>j$, after cars $C_1,\dots,C_{i-1}$ have parked, there remains an empty spot before position $\beta_j$. For cars $C_{i+1},\dots,C_n$, they first drive to their preferred spot which is after $\beta_j$, and go forward if it has already been occupied. Therefore, no car will park at the empty spot before $j$, i.e., $\alpha$ is not a parking function.

    Conversely, suppose that $\beta_i\le i$ for all $i\in[n]$. It is trivial to see that $\beta$ is a parking function. By \Cref{parking function rearrangement}, we have $\alpha$ is a parking function.
\end{proof}

\subsection{Primitive parking functions}

We define a parking function to be \textit{primitive} if it arranges in the weakly increasing order. For example, $(1,2,3,3)$ is a primitive parking function, whereas $(3,2,1,3)$ is a parking function that is not primitive. To determine the number of primitive parking functions of length $n$, we utilize the bijection between primitive parking functions and Dyck paths. A \textit{Dyck path} of length $n$ is a path from $(0,0)$ to $(n,n)$ consisting of \textit{east steps} $(1,0)$ and \textit{north steps} $(0,1)$, such that the path never goes below the diagonal $y=x$.

\begin{theorem}
    The number of Dyck path of order $n$ is given by the Catalan number
    \[C_n=\frac{1}{n+1}{2n\choose n}.\]
\end{theorem}

\begin{proof}
    Consider the set of paths from $(0,0)$ to $(n,n)$ consisting of east steps and north steps (with no restrictions). To get such a path, we must take $n$ steps east and $n$ steps north, so there are $\binom{2n}{n}$ paths from $(0,0)$ to $(n,n)$. We say a path from $(0,0)$ to $(n,n)$ is \textit{bad} if it crosses below the diagonal $y=x$ at any point. Let $B_n$ denote the set of all bad paths from $(0,0)$ to $(n,n)$. Hence, we want to show that
    \[C_n=\binom{2n}{n}-|B_n|.\]

    Now, consider any bad path and let $P=(i+1, i)$ be the first point where it crosses below $y=x$, where $1 \le i \le n-1$. See \Cref{fig 2a}. From $P$, we need $n-i$ east steps and $n-(i+1)$ north steps to reach $(n, n)$. If, instead, we \textit{reflect} this portion and take $n-i$ north steps and $n-(i+1)$ east steps from $P$, we will arrive at $(n+1, n-1)$ instead of $(n, n)$. See \Cref{fig 2b}. Conversely, for any path from $(0,0)$ to $(n+1,n-1)$, it must contain a point strictly below the diagonal. Let $P=(i+1,i)$ be such a point where $i$ is the smallest possible choice, where $1\le i \le n-1$. From $P$, if we reflect and take $n-i$ steps east and $n-(i+1)$ steps north, the path will end at $(n,n)$. Therefore, we conclude that there is a bijection between all paths from $(0,0)$ to $(n+1,n-1)$ and bad paths from $(0,0)$ to $(n,n)$.
    
    \begin{figure}[h]
    \captionsetup{justification=centerlast} % Center the caption
    \centering
    
    % First subfigure with the first circle
    \begin{subfigure}[b]{0.45\linewidth}
        \centering
        \begin{tikzpicture}[scale=0.75]
            \draw[step=1cm, gray] (0,0) grid (5,4);
            % Draw the diagonal line from (0,0) to (4,4)
            \draw[thick, gray] (0,0) -- (4,4);
            \draw[thick, dashed, gray] (1,0) -- (5,4);
        
            % Draw the path
            \draw[line width=1.5mm, wakegold] 
                (0,0) -- (0,2) -- (3,2) -- (3,4) -- (4,4);
            \fill[wakegold] (3,2) circle (1.5mm); % Radius is 1.5mm to get 3mm diameter
            
            % Add labels for clarity
            \node[below left] at (0,0) {\scriptsize $(0,0)$};
            \node[above] at (4,4) {\scriptsize $(4,4)$};
            \node[below right] at (3,2) {\scriptsize $P=(3,2)$};
        \end{tikzpicture}
        \caption{a bad path in the $4\times 4$ grid}
        \label{fig 2a}
    \end{subfigure}
    \hfill
    % Second subfigure with the second circle
    \begin{subfigure}[b]{0.45\linewidth}
        \centering
        \begin{tikzpicture}[scale=0.75]
            \draw[step=1cm, gray] (0,0) grid (5,4);
            \draw[thick, gray] (0,0) -- (4,4);
            \draw[thick, dashed, gray] (1,0) -- (5,4);
            \draw[line width=1.5mm, wakegold] 
                (0,0) -- (0,2) -- (5,2) -- (5,3);
            \fill[wakegold] (3,2) circle (1.5mm); % Radius is 1.5mm to get 3mm diameter
            
            \node[below left] at (0,0) {\scriptsize (0,0)};
            \node[above right] at (5,3) {\scriptsize (5,3)};
            \node[below right] at (3,2) {\scriptsize $P=(3,2)$};
        \end{tikzpicture}
        \caption{the reflected path}
        \label{fig 2b}
    \end{subfigure}

    \caption{Path (A) maps to (B) under the bijection, vice versa.}
    \end{figure}
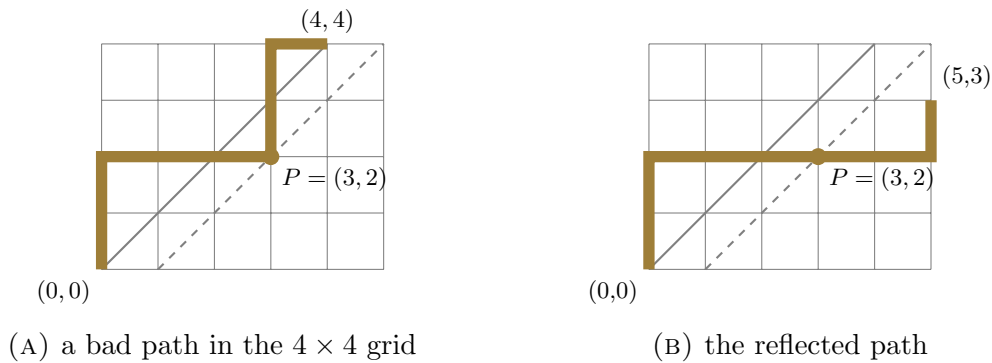

    From $(0,0)$ to $(n+1,n-1)$, there are $\binom{n+1 + (n-1)}{n+1}=\binom{2n}{n+1}$ distinct paths. Hence, we have
    \[C_n={2n\choose n}-{2n\choose{n+1}}=\frac{1}{n+1}{2n\choose n}.\qedhere\]
\end{proof}

The Catalan numbers count a wide range of combinatorial objects. Another example counted by the Catalan numbers appears in \Cref{prop: noncrossing partition and Catalan}. For a more comprehensive overview of sets counted by the Catalan numbers, see \cite{stanley_catalan}.

Dyck paths can be \textit{labeled}. Let $\alpha=(\alpha_1,\dots,\alpha_n)$ be a parking function and $\beta=(\beta_1,\dots,\beta_n)$ be its weakly increasing rearrangement. For each $i\in[n]$, let $k_i$ denote the number of occurrences of $i$ in $\beta$. Starting at $(0,0)$, move $k_i$ steps north for each $i$. Each time after completing the northward steps, move one step east. Note that this creates a Dyck path of length $n$. To label this Dyck path, let $j_1,j_2,\dots$ denote the positions in $\alpha$ where $i$ appears. Assign the labels $j_1, j_2, \dots$ to the vertical steps of the Dyck path in column $i$, ordering them from the bottom to the top.
\begin{comment}
\begin{enumerate}
    \item[a.] Starting at $(0,0)$, 
    we go $k_1$ steps north, where $k_1$ is the number of occurrences of $1$ in $\beta$; then move one step to the east.
    \item[b.] Go $k_2$ steps north, where $k_2$ is the number of occurrences of $2$ in $\beta$; then move one step to the east.
    \item[c.] Repeat above process until $k_n$, where $k_n$ is the number of occurrences of $n$ in $\beta$; then move one step to the east, and hit $(n,n)$.
    \item[d.] We put the labels inside the grids below the Dyck path. Each number of $[n]$ are labeled exactly once. From bottom to top, we put $x_1,x_2,\dots,x_l$ in each grid of the $i$-th column, where $x_j$ denotes the position of number $i$ in $\alpha$, and $i,j\in[n]$.
\end{enumerate}    
\end{comment}
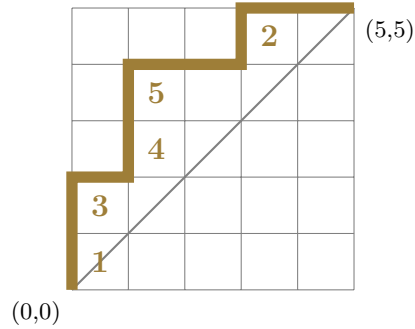
\begin{figure}[h]
    \centering
    \begin{tikzpicture}[scale=0.75]
        % Draw the grid (5x5 with 1cm cells)
        \draw[step=1cm, gray] (0,0) grid (5,5);
        \draw[thick, gray] (0,0) -- (5,5);
        
        % Draw the path
        \draw[line width=1.5mm, wakegold] 
            (0,0) -- (0,2) -- (1,2) -- (1,4) -- (2,4) -- (3,4) -- (3,5) -- (4,5) -- (5,5);
        
        % Add labels for clarity
        \node[below left] at (0,0) {\scriptsize (0,0)};
        \node[below right] at (5,5) {\scriptsize (5,5)};
        
        % Add labels
        \node[wakegold, font=\bfseries, align=center] at (0.5,0.5) {1}; 
        \node[wakegold, font=\bfseries, align=center] at (0.5,1.5) {3};
        \node[wakegold, font=\bfseries, align=center] at (1.5,2.5) {4};
        \node[wakegold, font=\bfseries, align=center] at (1.5,3.5) {5};
        \node[wakegold, font=\bfseries, align=center] at (3.5,4.5) {2};
    \end{tikzpicture}
    \caption{The labeled Dyck path of $\alpha=(1,4,1,2,2)$.}
\end{figure}

Conversely, given a labeled Dyck path, we can reconstruct the parking function $\alpha = (\alpha_1, \dots, \alpha_n)$ by assigning to each position $j$ in $\alpha$ the label $i$ if $j$ appears in column $i$ of the labeled Dyck path.

For each $i\in[n]$, the Dyck path staying above the diagonal after the vertical steps in column $i$, ensures that there are at least $i$ entries in $\alpha$ such that they are smaller or equal than $i$. Therefore, if we sort $\alpha$ into $\beta$, we have $\beta_i\le i$. Following the same reasoning, given a parking function, we can always get a Dyck path. Moreover, different labelings of Dyck path serve as an analogy to different permutations of $\alpha$. Therefore, we have a bijection between
\begin{enumerate}
    \item[a.] Dyck paths of length $n$ and parking functions of length $n$; and
    \item[b.] Labeled Dyck paths of length $n$ and primitive parking functions of length $n$.
\end{enumerate}

\section{Parking Functions and Posets}\label{section 2}
Parking functions possess a rich combinatorial structure of posets through their relationship with noncrossing partitions. In 1997, Richard P. Stanley showed that the number of maximal chains of noncrossing partitions of $[n+1]$ coincides with the number of parking functions of length $n$ providing a bijective correspondence. The primary objective of this section is to explore this bijection.

We first introduce several key definitions. A \textit{partition of a finite set} $S$ is a collection $\{B_1,B_2,\dots,B_k\}$ of nonempty pairwise disjoint subset $B_i\subseteq S$ such that $B_1\cup B_2\cup\dots\cup B_k=S$, where $i\in[k]$. A \textit{noncrossing partition} of $[n]$ is a partition $\{B_1,B_2,\dots,B_k\}$ of $[n]$ such for $a<b<c<d$, where $a,c\in B_i$, and $b,d\in B_j$, we have $i=j$. We use $\NC_n$ to denote the set of all noncrossing partitions of $[n]$, and it is a Catalan object.

\begin{theorem}[\cite{olds19494277}]\label{prop: noncrossing partition and Catalan}
    The number of noncrossing partitions of $[n]$ is the $n$-th Catalan number.
\end{theorem}
% Figure of noncrossing partition
\begin{figure}[h]
    \centering
    \begin{tikzpicture}
        % Circle and dot radius
        \newcommand{\circleRadius}{1} % Circle radius
        \newcommand{\dotRadius}{0.8}  % Radius for the 11 dots

        % 1
        \begin{scope}[shift={(0, 0)}]
            \draw circle (\circleRadius);
            % Draw the 11 dots within the circle, labeled in clockwise order
            \foreach \y in {1,...,11} {
                \fill (90 - \y*360/11:\dotRadius) circle (0.04);
            }
            % Define blocks: {1,2,3,4} and {8,11} in wakegold
            \fill[opacity=1, wakegold] 
                (90:\dotRadius) -- (58.18:\dotRadius) -- (26.36:\dotRadius) -- (-5.45:\dotRadius) -- cycle; % Block {1,2,3,4}
            \draw[opacity=1, wakegold, line width=3pt] 
                (90 - 7*360/11:\dotRadius) -- (90 - 10*360/11:\dotRadius); % Line segment for block {8,11}
        \end{scope}

        % 2
        \begin{scope}[shift={(2.4, 0)}]
            \draw circle (\circleRadius);
            % Draw the 11 dots within the circle, labeled in clockwise order
            \foreach \y in {1,...,11} {
                \fill (90 - \y*360/11:\dotRadius) circle (0.04);
            }
            % Define block: {1,2,3,7}
            \fill[opacity=1, wakegold] 
                (90:\dotRadius) -- (58.18:\dotRadius) -- (26.36:\dotRadius) -- (90 - 6*360/11:\dotRadius) -- cycle; % Block {1,2,3,7}
        \end{scope}

        % 3
        \begin{scope}[shift={(2*2.4, 0)}]
            \draw circle (\circleRadius);
            % Draw the 11 dots within the circle, labeled in clockwise order
            \foreach \y in {1,...,11} {
                \fill (90 - \y*360/11:\dotRadius) circle (0.04);
            }
            % Define blocks: {1,2,3,4,5,6,7,8,9} and {10,11}
            \fill[opacity=1, wakegold] 
                (90:\dotRadius) -- (58.18:\dotRadius) -- (26.36:\dotRadius) -- (90 - 3*360/11:\dotRadius) -- (90 - 4*360/11:\dotRadius) -- (90 - 5*360/11:\dotRadius) -- (90 - 6*360/11:\dotRadius) -- (90 - 7*360/11:\dotRadius) -- (90 - 8*360/11:\dotRadius) -- cycle; % Block {1,2,3,7}
            \draw[opacity=1, wakegold, line width=3pt] 
                (90 - 9*360/11:\dotRadius) -- (90 - 10*360/11:\dotRadius); % Line segment for block {10,11}
        \end{scope}
        
        % 4
        \begin{scope}[shift={(3*2.4, 0)}]
            \draw circle (\circleRadius);
            % Draw the 11 dots within the circle, labeled in clockwise order
            \foreach \y in {1,...,11} {
                \fill (90 - \y*360/11:\dotRadius) circle (0.04);
            }
            % Define blocks: {1,2,3,7} and {4,9} in purple_opposite
            \fill[opacity=.5, purple_opposite] 
                (90:\dotRadius) -- (58.18:\dotRadius) -- (26.36:\dotRadius) -- (90 - 6*360/11:\dotRadius) -- cycle; % Block {1,2,3,7}
            \draw[opacity=.5, purple_opposite, line width=3pt] 
                (90 - 3*360/11:\dotRadius) -- (90 - 8*360/11:\dotRadius); % Line segment for block {4,9}
        \end{scope}

        % 5
        \begin{scope}[shift={(4*2.4, 0)}]
            \draw circle (\circleRadius);
            % Draw the 11 dots within the circle, labeled in clockwise order
            \foreach \y in {1,...,11} {
                \fill (90 - \y*360/11:\dotRadius) circle (0.04);
            }
            % Define blocks: {4,9,10} and {2,7} in purple_opposite
            \fill[opacity=.5, purple_opposite] 
                (90 - 3*360/11:\dotRadius) -- (90 - 8*360/11:\dotRadius) -- (90 - 9*360/11:\dotRadius) -- cycle; % Block {4,9,10}
            \draw[opacity=.5, purple_opposite, line width=3pt] 
                (90 - 1*360/11:\dotRadius) -- (90 - 6*360/11:\dotRadius); % Line segment for block {2,7}
        \end{scope}

        % 6
        \begin{scope}[shift={(5*2.4, 0)}]
            \draw circle (\circleRadius);
            % Draw the 11 dots within the circle, labeled in clockwise order
            \foreach \y in {1,...,11} {
                \fill (90 - \y*360/11:\dotRadius) circle (0.04);
            }
            % Define blocks: {1,2,3,9} and {4,7,11} in purple_opposite
            \fill[opacity=.5, purple_opposite] 
                (90:\dotRadius) -- (90 - 1*360/11:\dotRadius) -- (90 - 2*360/11:\dotRadius) -- (90 - 8*360/11:\dotRadius)-- cycle; % Block {1,2,3,9}
            \fill[opacity=.5, purple_opposite] 
                (90 - 3*360/11:\dotRadius) -- (90 - 6*360/11:\dotRadius) -- (90 - 10*360/11:\dotRadius)-- cycle; % Block {4,7,11}
        \end{scope}
        
    \end{tikzpicture}
    \caption{noncrossing partitions (first three on the left) \\ and crossing partitions (last three on the right) of $[11]$.}
\end{figure}
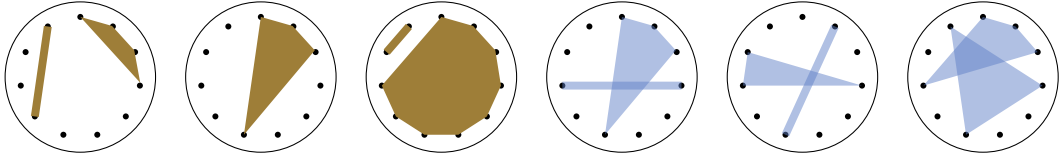

A fundamental property of the set of noncrossing partitions of $[n]$ is that it can be given a natural partial ordering. A \textit{partially ordered set} or \textit{poset}, is a pair $(P,\le)$ where $P$ is a set and ``$\le$" is a binary relation on $P$ satisfying 
\begin{enumerate}
    \item[a.] (reflexivity) $x\le x$,
    \item[b.] (antisymmetry) if $x\le y$ and $y\le x$, then $x=y$, and
    \item[c.] (transitivity) if $x\le y$ and $y\le z$, then $x\le z$,
\end{enumerate}
for all $x,y,z\in P$. $\NC_n$ can be given a natural partial order such that $\pi\le\sigma$ if every block of $\pi$ is contained in a block of $\sigma$, i.e., $\pi$ refines $\sigma$.

Often a poset is represented by a certain graph which can be easier to work with than just using above axioms. For $x,y\in P$, if $x\le y$ and $x\neq y$, we write $x<y$. We say $x$ is \textit{covered} by $y$ (or $y$ \textit{covers} $x$), written either $x\lessdot y$ or $y\gtrdot x$, if $x<y$ and there is no $z\in P$ with $x<z<y$. The \textit{Hasse diagram} of $P$ is the graph with vertices $P$ and an edge from $x$ up to $y$ if $x\lessdot y$. \Cref{Hasse diagram for the noncrossing partition of [3]} is the Hasse diagram for the noncrossing partition of $[3]$. Observe that two noncrossing partitions are at the same \textit{level} in the Hasse diagram if they are incomparable with respect to the partial order defined above; moreover, they partition $[n]$ into the same number of blocks.

% Figure of Hasse diagram of noncrossing partitions of [3].
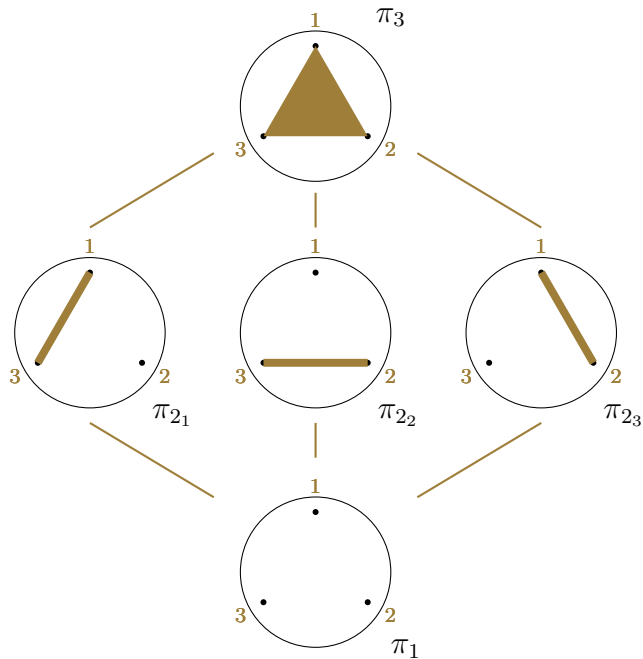
\begin{figure}[h]
    \centering
    \begin{tikzpicture}
        % Circle and dot radius
        \newcommand{\circleRadius}{1} % Circle radius
        \newcommand{\dotRadius}{0.8}  % Radius for the 3 dots

        % Top level: {1,2,3}
        \begin{scope}[shift={(0, 3.2)}]
            \node (Q1) at (0,0) {
            \begin{tikzpicture}
            \draw circle (\circleRadius);
            % Draw the 3 dots within the circle
            \foreach \y in {1,2,3} {
                \fill (90 - \y*360/3:\dotRadius) circle (0.04);
                \node at (90:\dotRadius+.35) {\tiny {\textcolor{wakegold}{1}}};
                \node at (90 - 120:\dotRadius+.35) {\tiny {\textcolor{wakegold}{2}}};
                \node at (90 - 240:\dotRadius+.35) {\tiny {\textcolor{wakegold}{3}}};
            }
            % Define block {1,2,3} in wakegold
            \fill[opacity=1, wakegold] 
                (90:\dotRadius) -- (-30:\dotRadius) -- (210:\dotRadius) -- cycle;
             \end{tikzpicture}};
            \node at (1,1) {$\pi_3$}; % Label
        \end{scope}

        % Middle level: {1,2}, {2,3}, {1,3}
        \begin{scope}[shift={(-3, 0)}]
            \draw circle (\circleRadius);
            \foreach \y in {1,2,3} {
                \fill (90 - \y*360/3:\dotRadius) circle (0.04);
                \node at (90:\dotRadius+.35) {\tiny {\textcolor{wakegold}{1}}};
                \node at (90 - 120:\dotRadius+.35) {\tiny {\textcolor{wakegold}{2}}};
                \node at (90 - 240:\dotRadius+.35) {\tiny {\textcolor{wakegold}{3}}};
            }
            % Block {1,2} in wakegold
            \draw[opacity=1, wakegold, line width=3pt] 
                (90:\dotRadius) -- (210:\dotRadius);
            \node at (1.1, -1.1) {$\pi_{2_1}$}; % Label
        \end{scope}

        \begin{scope}[shift={(0, 0)}]
            \draw circle (\circleRadius);
            \foreach \y in {1,2,3} {
                \fill (90 - \y*360/3:\dotRadius) circle (0.04);
                \node at (90:\dotRadius+.35) {\tiny {\textcolor{wakegold}{1}}};
                \node at (90 - 120:\dotRadius+.35) {\tiny {\textcolor{wakegold}{2}}};
                \node at (90 - 240:\dotRadius+.35) {\tiny {\textcolor{wakegold}{3}}};
            }
            % Block {2,3} in wakegold
            \draw[opacity=1, wakegold, line width=3pt] 
                (-30:\dotRadius) -- (210:\dotRadius);
            \node at (1.1, -1.1) {$\pi_{2_2}$}; % Label
        \end{scope}

        \begin{scope}[shift={(3, 0)}]
            \draw circle (\circleRadius);
            \foreach \y in {1,2,3} {
                \fill (90 - \y*360/3:\dotRadius) circle (0.04);
                \node at (90:\dotRadius+.35) {\tiny {\textcolor{wakegold}{1}}};
                \node at (90 - 120:\dotRadius+.35) {\tiny {\textcolor{wakegold}{2}}};
                \node at (90 - 240:\dotRadius+.35) {\tiny {\textcolor{wakegold}{3}}};
            }
            % Block {1,3} in wakegold
            \draw[opacity=1, wakegold, line width=3pt] 
                (90:\dotRadius) -- (-30:\dotRadius);
            \node at (1.1, -1.1) {$\pi_{2_3}$}; % Label
        \end{scope}

        % Bottom level: {1}, {2}, {3}
        \begin{scope}[shift={(0, -3)}]
        \node (Q2) at (0,0) {
        \begin{tikzpicture}
            \draw circle (\circleRadius);
            \foreach \y in {1,2,3} 
            {
            \fill (90 - \y*360/3:\dotRadius) circle (0.04);
            \node at (90:\dotRadius+.35) {\tiny {\textcolor{wakegold}{1}}};
                \node at (90 - 120:\dotRadius+.35) {\tiny {\textcolor{wakegold}{2}}};
                \node at (90 - 240:\dotRadius+.35) {\tiny {\textcolor{wakegold}{3}}};
            }
        \end{tikzpicture}};
        \node at (1.2,-1.2) {$\pi_1$}; % Label
        \end{scope}

        % Connections
        \draw[thick, wakegold] (Q1) -- (-3, 1.4);  % Connect top to {1,2}
        \draw[thick, wakegold] (Q1) -- (0, 1.4);   % Connect top to {2,3}
        \draw[thick, wakegold] (Q1) -- (3, 1.4);   % Connect top to {1,3}
        
        \draw[thick, wakegold] (-3, -1.2) -- (Q2); % Connect {1,2} to {1}, {2}, {3}
        \draw[thick, wakegold] (0, -1.2) -- (Q2);   % Connect {2,3} to {1}, {2}, {3}
        \draw[thick, wakegold] (3, -1.2) -- (Q2);   % Connect {1,3} to {1}, {2}, {3}
    \end{tikzpicture}
    \caption{Hasse diagram for $\NC_3$.}
    \label{Hasse diagram for the noncrossing partition of [3]}
\end{figure}

A \textit{maximal chain} $\mathfrak{m}$ of $\NC_n$ is a sequence of noncrossing partitions $\pi_1,\pi_2,\dots,\\\pi_n$ such that $\pi_i\le\pi_{i+1}$ for each $i\in[n-1]$, $\pi_{i+1}$ can be obtained from $\pi_{i}$ by merging two blocks of $\pi_{i}$ into a single block. In \Cref{Hasse diagram for the noncrossing partition of [3]}, $\pi_1, \pi_{2_1},\pi_3$ is a maximal chain of $\NC_3$. In 1980, Paul Edelman \cite[Corollary 3.3]{Edelman} showed that the $\NC_{n+1}$ has $(n+1)^{n-1}$ maximal chains, which coincides the number of parking functions of length $n$ [\Cref{cardinality of pf}]. This motivates Stanley's theorem. 
\begin{theorem}[{\cite[Theorem 3.1]{PFandNoncrossing}}]\label{pf and noncrossing}
    There is a bijection between parking functions of length $n$ and maximal chains of noncrossing partitions of $[n+1]$.
\end{theorem}

We describe Stanley's bijection.  First, a maximal chain of $\NC_{n+1}$ can be thought of a series of steps starting with the finest partition $\pi_1=\{1\},\{2\},\dots,\\\{n+1\}$ and then merging two blocks together one at a time at each step until we reach $\pi_{n+1}=\{1,2,\dots,n+1\}$. For example, consider the maximal chain $\mathfrak{m}=\pi_1, \pi_{2_1},\pi_3$ in \Cref{Hasse diagram for the noncrossing partition of [3]}. We merge the blocks $\{1\}$ and $\{3\}$ of $\pi_1$ at step $1$, merge the blocks $\{1,3\}$ and $\{2\}$ of $\pi_{2_1}$ at step $2$, and we reach $\pi_3=\{1,2,3\}$. Stanley's bijection then gives labels $\alpha_1,\dots,\alpha_n$ to each step: Let $A$ and $B$ be the two blocks of $\pi_i$ we are going to merge at step $i$, where $A$ contains the smallest element in the disjoint union $A\cup B$. The label $\alpha_i$ for this step is the largest element in $A$ which is smaller than all elements in $B$. The sequence of labels $\alpha=(\alpha_1,\alpha_2,\dots,\alpha_n)$ forms a parking function of length $n$. 

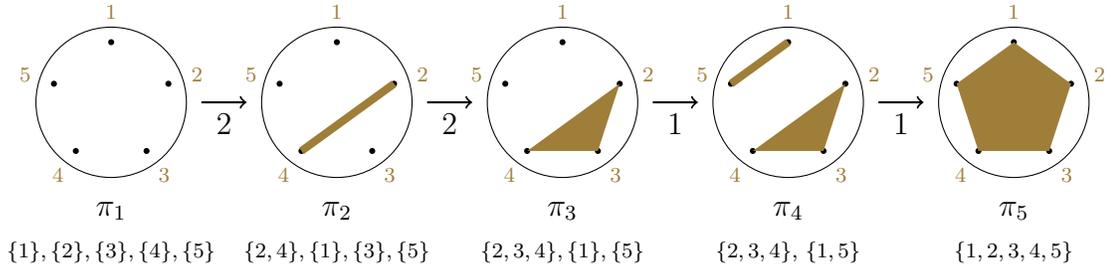
\begin{figure}[h]
    \centering
    \begin{tikzpicture}
        % Circle and dot radius
        \newcommand{\circleRadius}{1} % Circle radius
        \newcommand{\dotRadius}{0.8}  % Radius for the 5 dots

        % 1
         \begin{scope}[shift={(0, 0)}]
            \draw circle (\circleRadius);
            \foreach [count=\i from 0] \y in {1,2,3,4,5} {
                \fill (90 - \i*360/5:\dotRadius) circle (0.04);
                \node at (90 - \i*360/5:\dotRadius+.4) {\tiny {\textcolor{wakegold}{\y}}};
            }
            \node[below=5pt] at (0, -\circleRadius) {$\pi_1$}; % Empty set label, smaller font
            \node[below=20pt] at (0, -\circleRadius) {\tiny $\{1\}, \{2\}, \{3\}, \{4\}, \{5\} $}; % Empty set label, smaller font
        \end{scope}

        % 2
        \begin{scope}[shift={(3, 0)}]
            \draw circle (\circleRadius);
            \foreach [count=\i from 0] \y in {1,2,3,4,5} {
                \fill (90 - \i*360/5:\dotRadius) circle (0.04);
                \node at (90 - \i*360/5:\dotRadius+.4) {\tiny {\textcolor{wakegold}{\y}}};
            }
            % Block {2,4} in wakegold
            \draw[opacity=1, wakegold, line width=3pt] 
            (90 - 1*360/5:\dotRadius) -- (90 - 3*360/5:\dotRadius); 
            \node[below=5pt] at (0, -\circleRadius) {$\pi_2$}; % Label
            \node[below=20pt] at (0, -\circleRadius) {\tiny $\{2,4\}, \{1\}, \{3\}, \{5\}$};
        \end{scope}

        % 3
        \begin{scope}[shift={(6, 0)}]
            \draw circle (\circleRadius);
            \foreach [count=\i from 0] \y in {1,2,3,4,5} {
                \fill (90 - \i*360/5:\dotRadius) circle (0.04);
                \node at (90 - \i*360/5:\dotRadius+.4) {\tiny {\textcolor{wakegold}{\y}}};
            }
            % Block {2,3,4} in wakegold
            \fill[opacity=1, wakegold] 
                (18:\dotRadius) -- (-54:\dotRadius) -- (90 - 3*360/5:\dotRadius) -- cycle;
            \node[below=5pt] at (0, -\circleRadius) {$\pi_3$}; % Label
            \node[below=20pt] at (0, -\circleRadius) {\tiny $\{2,3,4\}, \{1\}, \{5\}$}; % Label
        \end{scope}

        % 4
        \begin{scope}[shift={(9, 0)}]
            \draw circle (\circleRadius);
            \foreach [count=\i from 0] \y in {1,2,3,4,5} {
                \fill (90 - \i*360/5:\dotRadius) circle (0.04);
                \node at (90 - \i*360/5:\dotRadius+.4) {\tiny {\textcolor{wakegold}{\y}}};
            }
            % Block {1,5} in wakegold
            \draw[opacity=1, wakegold, line width=3pt] 
            (90:\dotRadius) -- (90 - 4*360/5:\dotRadius); 

            \fill[opacity=1, wakegold] 
                (18:\dotRadius) -- (-54:\dotRadius) -- (90 - 3*360/5:\dotRadius) -- cycle;
            \node[below=5pt] at (0, -\circleRadius) {$\pi_4$}; % Label
            \node[below=20pt] at (0, -\circleRadius) {\tiny $\{2,3,4\}$, $\{1,5\}$}; % Label
        \end{scope}

        % 5
        \begin{scope}[shift={(12, 0)}]
            \draw circle (\circleRadius);
            \foreach [count=\i from 0] \y in {1,2,3,4,5} {
                \fill (90 - \i*360/5:\dotRadius) circle (0.04);
                \node at (90 - \i*360/5:\dotRadius+.4) {\tiny {\textcolor{wakegold}{\y}}};
            }
            % Block {1,2,3,4,5} in wakegold
            \fill[opacity=1, wakegold] 
                (90:\dotRadius) -- (18:\dotRadius) -- (-54:\dotRadius) -- (-126:\dotRadius) -- (-198:\dotRadius) -- cycle;
            \node[below=5pt] at (0, -\circleRadius) {$\pi_5$}; % Label
            \node[below=20pt] at (0, -\circleRadius) {\tiny $\{1,2,3,4,5\}$}; % Label
        \end{scope}

        % Arrows with labels
        \draw[->, thick] (1.2, 0) -- (1.8, 0) node[midway, below] {2};
        \draw[->, thick] (4.2, 0) -- (4.8, 0) node[midway, below] {2};
        \draw[->, thick] (7.2, 0) -- (7.8, 0) node[midway, below] {1};
        \draw[->, thick] (10.2, 0) -- (10.8, 0) node[midway, below] {1};

    \end{tikzpicture}
    \caption{One of the maximal chain of $\{1,2,3,4,5\}$ which is associated with parking function $(2,2,1,1)$.}
    \label{maximal chain and pf}
\end{figure}

For example, consider some maximal chain of $\NC_5$ illustrated in \Cref{maximal chain and pf} above. Applying the algorithm, at step $4$, we are about to merge the blocks $\{2,3,4\}$ and $\{1,5\}$ of $\pi_4$. Let $A=\{1,5\}$ and $B=\{2,3,4\}$, where $1\in A$ is the smallest element in $A\cup B=\{1,2,3,4,5\}$. We pick $1$ as the label of this step, since $1\in A$ is the largest element smaller than all element in $B$.

\begin{proof}[Proof sketch of \Cref{pf and noncrossing}]
    Suppose $\pi_{j+1}$ is obtained from $\pi_j$ by merging together two blocks $A$ and $B$ of $\pi_j$, where $\min A<\min B$ and $j\in[n]$. Define
    \[\alpha_j\coloneqq\Lambda(\pi_j,\pi_{j+1})\coloneqq\max\{i\in A:i<\min B\}.\]
    
    First we want to show that $\alpha=(\alpha_1,\dots,\alpha_n)$ is a parking function. Consider the label $\alpha_j=\Lambda(\pi_j,\pi_{j+1})$ for step $j$. Note that $\alpha_j<\min B$, which means that if $\alpha_j$ is the label of some step, the merged block of $\pi_{j+1}$ containing $\alpha_j$ also contains an element $k>\alpha_j$, $k\in[n]$. This ensures that the number of occurrences of $\alpha_j$ in $\alpha$ is at most $n+1-\alpha_j$. If $\alpha_j$ appears exactly $n+1-\alpha_j$ times, it will occupy the $\alpha_j$-th to the $n$-th entries in the weakly increasing rearrangement of $\alpha$. Therefore, the leftmost $\alpha_j$ in the sorted sequence cannot exceed position $\alpha_j$. Thus, $\alpha$ is a parking function.

    To show uniqueness, let $\alpha_1$ and $\alpha_2$ be two sequences associated with maximal chains $\mathfrak{m}_1$ and $\mathfrak{m}_2$ of $\NC_{n+1}$, respectively. We induct on $n$ to show that if $\alpha_1=\alpha_2$, then $\mathfrak{m}_1=\mathfrak{m}_2$. For more details, see \cite[Theorem 3.1]{PFandNoncrossing}.

    Conversely, since the $\NC_{n+1}$ has $(n+1)^{n-1}$ maximal chains, and there are $(n+1)^{n-1}$ parking functions of length $n$, we conclude that every parking function of length $n$ occurs exactly once among $\alpha$.
\end{proof}

\section{Parking Function Polytopes}\label{section 3}

Following the definitions in \textit{Lectures on Polytope} by G\"unter Ziegler \cite{ziegler1995lectures}, a point set $K\subseteq \R^n$ is \textit{convex} if with any two points $\vb x,\vb y\in K$, it contains a straight line segment between them. The \textit{convex hull} of $K$ is the smallest convex set containing $K$. A (\textit{convex}) \textit{polytope} is the convex hull of a finite point set. Formally, the polytope formed by $K=\{\vb v_1,\vb v_2,\dots,\vb v_k\}\subseteq \R^n$ is
\[\conv(K)=\left\{\vb x\in\R^n: \vb x = \sum_{i=1}^k \lambda_i \vb v_i, \lambda_i \ge 0, \sum_{i=1}^k\lambda_i=1\right\}, \]
and we call such $\vb x$ is a \textit{convex combination} of $K$. 

Parking functions of length $n$ can be viewed as elements in $\R^n$. For example, the three parking functions of length $2$: $(1,1)$, $(1,2)$, and $(2,1)$, can be treated as points in $\R^2$, and the polytope formed by these three points is a triangle, see \Cref{parking function polytope PF_2}. In general, we define parking function polytope $\mathcal{PF}_n\in\R^n$ to be the convex hull of all parking functions of length $n$, i.e.,
\[\mathcal{PF}_n = \conv\left(\left\{ \alpha\in\R^n: \alpha \text{ is a parking function of length }n\right\}\right).\]
\Cref{parking function polytope PF_3} is the polytope in $\R^3$ formed by all parking functions of length $3$.
\begin{figure}[h]
    \captionsetup{justification=centerlast}
    \centering
    \begin{subfigure}[b]{0.45\linewidth}
        \includegraphics[width=\linewidth]{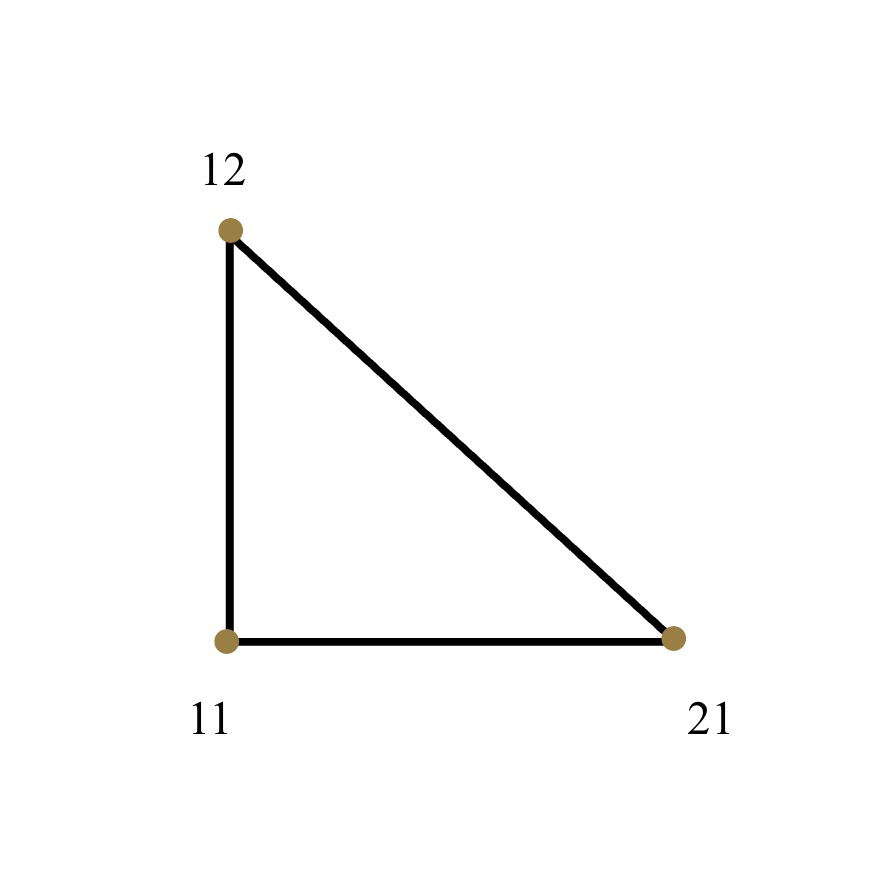}
        \caption{Parking function polytope $\mathcal{PF}_2$.}
        \label{parking function polytope PF_2}
    \end{subfigure}
    \hfill
    \begin{subfigure}[b]{0.45\linewidth}
        \includegraphics[width=\linewidth]{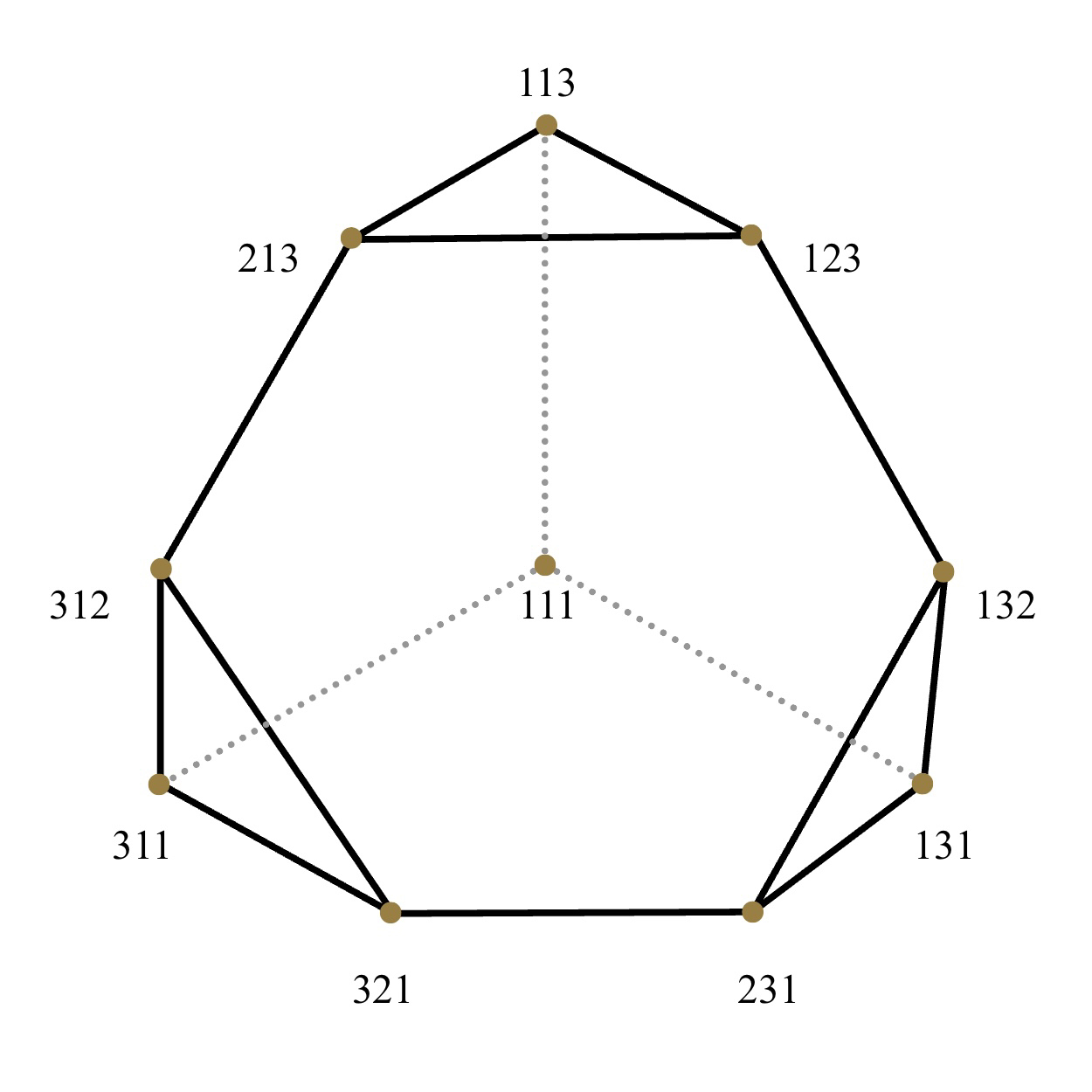}
        \caption{Parking function polytope $\mathcal{PF}_3$.}
        \label{parking function polytope PF_3}
    \end{subfigure}
    \caption{}
\end{figure}

We call $F$ a face of a polytope $\mathcal{P}$ if $F=\mathcal{P}\cap\{\vb x:\vb c\cdot \vb x=d\}$
for some $\vb c\in\R^n$ and $d\in\R$, such that for all $\vb x\in\mathcal{P}$, $\vb c\cdot \vb x\le d$, where the dot $\cdot$ means the dot product. We call a face a \textit{vertex} if it has dimension $0$, an \textit{edge} if it has dimension $1$, and a \textit{facet} if it has dimension $n-1$ given that $\mathcal{P}$ has dimension $n$. Understanding the number of faces of a polytope is crucial for analyzing its geometric structure. Each parking function of length $n$ lies on the boundary of $\mathcal{PF}_n$, meaning it is part of some face of $\mathcal{PF}_n$, whether that be a vertex, edge, facet, or higher-dimensional face.
For $\mathcal{PF}_3$, it has $7$ facets: one hexagonal facet, three pentagonal facets, and three triangular facets. The hexagonal facet, known as permutahedron $\Pi_3$, is particularly notable. 

According to Ziegler, it was first investigated by Schoute in 1911 \cite{Schoute1911}. The \textit{permutahedron} $\Pi_n\in \R^n$ is an $(n-1)$-dimensional polytope defined as the convex hull of all vectors that are obtained by permuting the coordinates of the vector $(1,2,\dots,n)$. Each of its vertices $(x_1,x_2,\dots,x_n)$ can be identified as a permutation in $\S_n$ via the map $x_i\mapsto i$, and two vertices are adjacent if and only if their corresponding permutations differ by an adjacent transposition.

% Permutahedron Pi_3
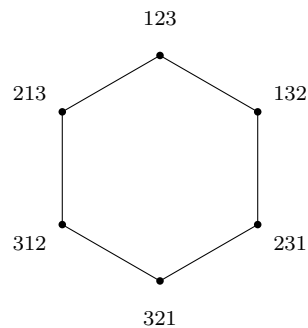
\begin{figure}[h]
    \centering
    \begin{tikzpicture}
        % Define the radius of the hexagon
        \newcommand{\radius}{1.5}

        % Define labels for each vertex with increased offset
        \foreach \i/\label in {1/213, 2/312, 3/321, 4/231, 5/132, 6/123} {
            % Position each dot
            \fill (90 + \i*60:\radius) circle (0.05);
            % Label each vertex with a larger offset
            \node at (90 + \i*60:\radius + 0.5) {\tiny \label};
        }

        % Connect the vertices to form a hexagon
        \foreach \i [evaluate=\i as \next using {int(mod(\i,6)+1)}] in {1,...,6} {
            \draw (90 + \i*60:\radius) -- (90 + \next*60:\radius);
        }
    \end{tikzpicture}
    \caption{Permutahedron $\Pi_3$.}
\end{figure}

We can use a Schlegel diagram to visualize $\mathcal{PF}_n$ in a lower-dimensional space. Loosely speaking, a \textit{Schlegel diagram} projects a polytope in $\R^n$ into $\R^{n-1}$ by projecting from a point just outside the polytope, preserving the structure of each faces. \Cref{Schelegel diagram PF3 only vertices} is the Schlegel diagram of $\mathcal{PF}_3$. We observe that not all parking functions are vertices of parking function polytopes. For example, the $10$ vertices of $\mathcal{PF}_3$ include six permutations of $(1,2,3)$, three permutations of $(1,1,3)$, and $(1,1,1)$. The remaining six parking functions of length $3$, consisting of the permutations of $(1,1,2)$ and $(1,2,2)$, either lie on the edges or facets of $\mathcal{PF}_3$, as shown in \Cref{Schelegel diagram PF3 not only vertices}. For instance, $(1,1,2)$ lies on the edge between $(1,1,1)$ and $(1,1,3)$, forming a linear sequence $(1,1,1)$, $(1,1,2)$, and $(1,1,3)$. 

\begin{figure}[h]
    \captionsetup{justification=centerlast}
    \centering
    \begin{subfigure}[b]{0.45\linewidth}
        \includegraphics[width=\linewidth]{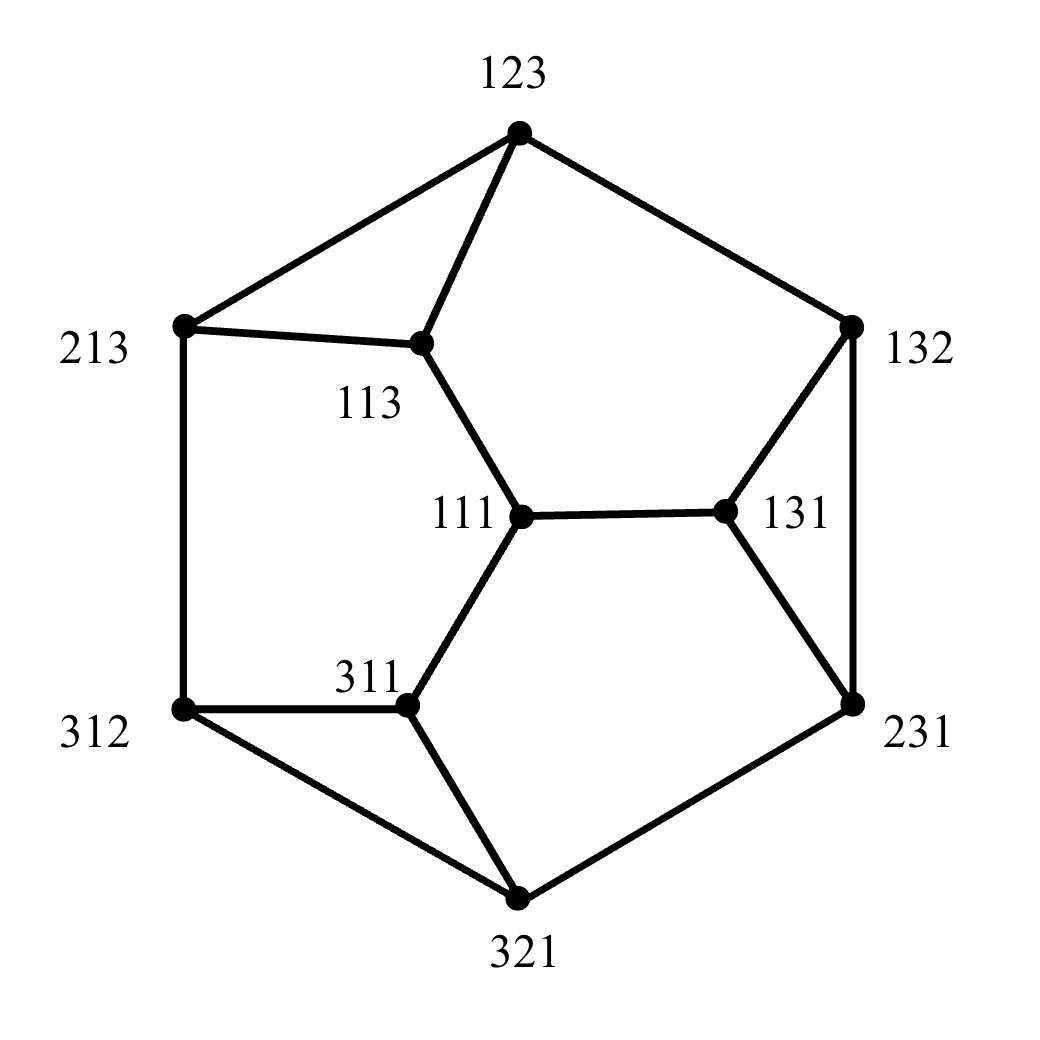}
        \caption{}
        \label{Schelegel diagram PF3 only vertices}
    \end{subfigure}
    \hfill
    \begin{subfigure}[b]{0.45\linewidth}
        \includegraphics[width=\linewidth]{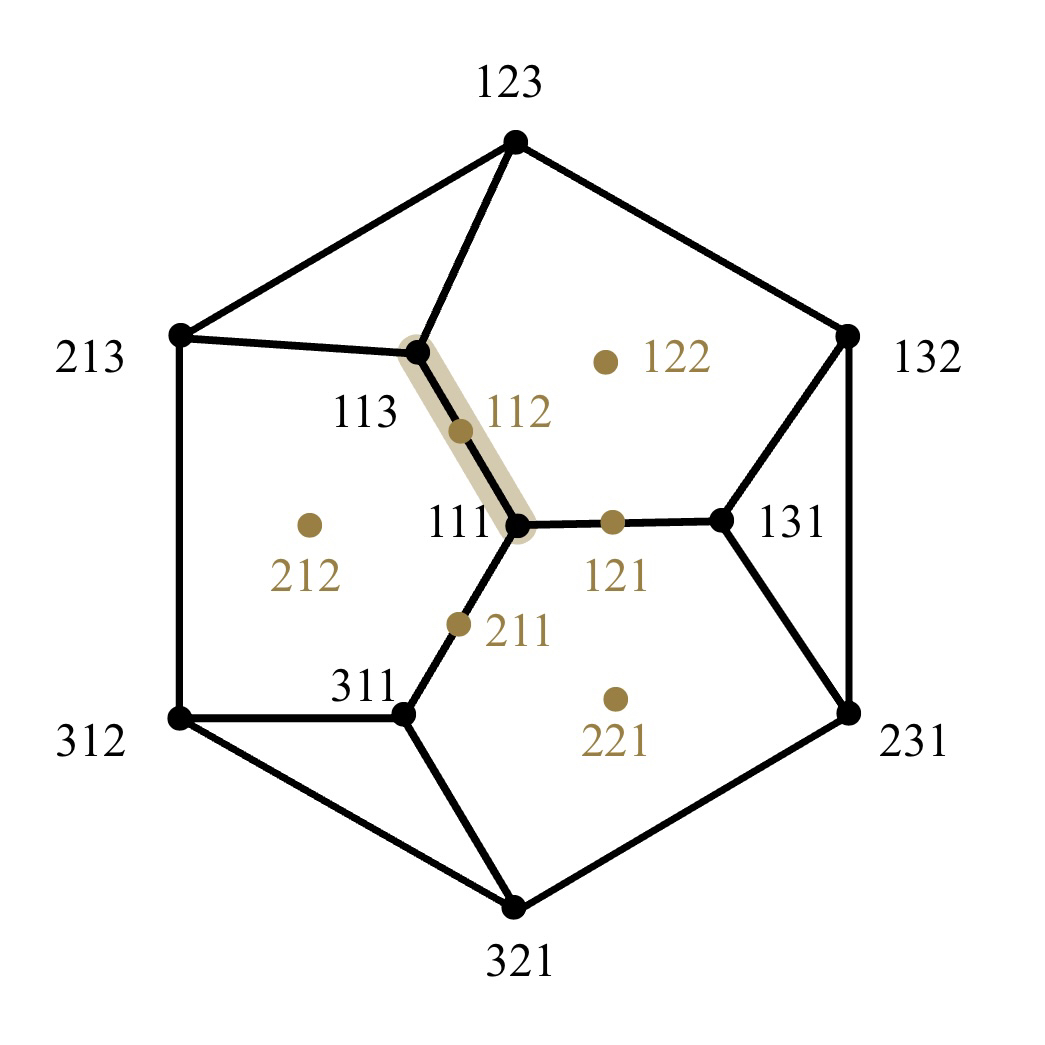}
        \caption{}
        \label{Schelegel diagram PF3 not only vertices}
    \end{subfigure}
    \caption{Schlegel diagram of $\mathcal{PF}_3$}
\end{figure}

These observations lead to two important general principles regarding the vertices of the parking function polytope:
\begin{enumerate}
    \item[a.] If a parking function is not a vertex of the polytope, increasing any entry greater than $1$ by $1$ yields another parking function (In other wards, it is possible to take a ``one-step" forward).
    \item[b.] If a parking function is a vertex of the polytope, then all of its rearrangements are also vertices.
\end{enumerate}

% First, if a parking function can be incremented in one of its positions to yield another valid parking function (that is, if it is possible to take a ``one-step" forward), then it is not a vertex of the polytope. Such a parking function lies within an edge or a higher-dimensional face, as it can be expressed as a convex combination of other parking functions. Second, if some parking function is a vertex, then all of its permutations are also vertices. 
The below proposition by Stanley formalizes these ideas.
%%This is a comment

\begin{prop}[{\cite[pg. 2]{ConvesHull2021}}]
    For each positive integer $n$, $\vb v=(v_1,\dots,v_n)$ is a vertex of $\mathcal{PF}_n$ if and only if it a rearrangement of 
    \[(\underbrace{1,\dots,1}_{k\text{ times }}, k+1, k+2, \dots, n),\]
    for some $1\le k\le n$.
\end{prop}
\begin{proof}
    First, for $k\in [n]$, observe that any permutation of the sequence
    \[\alpha^{(k)}=(\underbrace{1,\dots,1}_{k\text{ times }}, k+1, k+2, \dots, n)\]
    is a parking function of length $n$. Moreover, increasing any entry greater than $1$ in $\alpha^{(k)}$ by $1$ results in a sequence that is no longer a parking function. For example,
    \[(\underbrace{1,\dots,1}_{k\text{ times }}, k+2, k+2, \dots, n)\]
    is not a parking function.

    Suppose $\alpha=(\alpha_1,\alpha_2,\cdots,\alpha_n)$ is a parking function that is not a rearrangement of any $\alpha^{(k)}$. Then there exists an index $i$ such that for $\alpha_i>1$, we increase $\alpha_i$ by one yielding another parking function
    \[\alpha'=(\alpha_1,\dots,\alpha_{i-1}, \alpha_i+1, \alpha_{i+1},\dots,\alpha_n).\]
    This implies that $\alpha$ lies on the line segment connecting $\alpha'$ and
    \[\alpha''=(\alpha_1,\dots,\alpha_{i-1}, \alpha_i-1, \alpha_{i+1},\dots,\alpha_n).\]
    Therefore, $\alpha$ can be expressed as a convex combination
    \[\alpha=\frac{1}{2}\alpha'+\frac{1}{2}\alpha''.\]
    Thus, $\alpha$ is not a vertex of $\mathcal{PF}_n$.

    Conversely, suppose $\alpha$ is a permutation of $\alpha^{(k)}$ for some $k$. Suppose for contradiction that $\alpha$ is not a vertex of $\mathcal{PF}_n$. Then it can be written as a convex combination of two other parking functions $\beta,\gamma$, which are vertices of $\mathcal{PF}_n$, say
    \[\alpha=\lambda\beta+(1-\lambda)\gamma,\]
    for some $0<\lambda<1$. This implies that there exists an $i$ such that $\alpha_i\neq\beta_i$ and $\alpha_i\neq\gamma_i$. Without loss of generality, $\beta_i<\alpha_i<\gamma_i$ since $0<\lambda<1$. Each entry of a parking function must be a positive integer, so $\alpha_i\ge 2$. However, increasing any entry of $\alpha$ will not result in a parking function, so $\gamma$ is not a parking function, which is a contradiction.
\end{proof}

\begin{comment}
    \begin{proof}
    We first prove the forward direction. First, note that (1) is a parking function of length $n$. And if we increase any position greater than $1$ by $1$, (1) fails to be a parking function. For example,
    \[(\underbrace{1,\cdots,1}_{k\text{ times }}, k+2, k+2, \cdots, n)\]
    is not a parking function. Now, suppose that $\alpha=(\alpha_1,\alpha_2,\cdots,\alpha_n)$ is a parking function not of the form (1). That is to say, for some term $\alpha_i>1$, if we replace $\alpha_i$ by $\alpha_i+1$, we will still get a parking function, i.e., 
    \[\alpha'=(\alpha_1,\cdots,\alpha_{i-1}, \alpha_i+1, \alpha_{i+1},\cdots,\alpha_n)\] 
    is also a parking function. By the observation we made above, starting from $\alpha$, it is possible to ``take one-step forward". Therefore, $\alpha$ is not a vertex of $\mathcal{PF}_n$.
    
    Conversely, suppose that $\alpha$ is a parking function of the form (1). Seeking for a contradiction, we assume $\alpha$ is a not vertex of $\mathcal{PF}_n$. Therefore, $\alpha$ can be written as a convex combination of two other parking functions, say
    \begin{align*}
        \alpha =\lambda_1 \beta + \lambda_2 \gamma,
    \end{align*}
    where $\beta,\gamma\in\PF_n$, and $\lambda_1,\lambda_2\in [0,1]$. Since each term of a parking function is equal or greater than $1$, then we must have $\lambda_i = 0$ for $i=1$ or $2$. This implies that $\alpha=\beta$ or $\alpha=\gamma$, shows that $\alpha$ is a vertex of $\mathcal{PF}_n$.
\end{proof}
\end{comment}

To determine the number of vertices of $\mathcal{PF}_n$, we count the distinct permutations of parking function
\[(\underbrace{1,\dots,1}_{k\text{ times }}, k+1, k+2, \dots, n)\]
for each $k\in[n]$. The number of such permutation is given by the multinomial coefficient
\[{n\choose {k,\underbrace{1,\dots,1}_{\text{$n-k$ times}}}}=\frac{n!}{k!}.\]
Summing over $k$ from $1$ to $n$, the total number of vertices is
\[n!\left(\frac{1}{1!}+\frac{1}{2!}+\cdots+\frac{1}{n!}\right).\]
If the reader is interested in the number of edges and faces of higher dimensions of $\mathcal{PF}_n$, see \cite[section 2,3]{ConvesHull2021}. 

We end this paper by presenting one more example of parking function polytopes.

\begin{example}
    The Schlegel diagram of $\mathcal{PF}_4$ is shown in \Cref{fig 7-a}. $\mathcal{PF}_4$ has only one three-dimensional face, which is the convex hull of $24$ vertices, namely, the $24$ permutations of $(1,2,3,4)$. The convex hull of these permutations is exactly the permutahedron $\Pi_4$ (as shown in \Cref{fig 7-b}). Similarly, we can find $8$ two-dimensional faces, each of which are the convex hulls of $6$ vertices. These two-dimensional faces are $\Pi_3$.
    % $1123,1132,1213,1231,1312,1321$

    \begin{figure}[h]
    \captionsetup{justification=centerlast} % Center the caption
    \centering
    \begin{subfigure}[b]{0.45\linewidth}
        \centering
        \includegraphics[width=\linewidth]{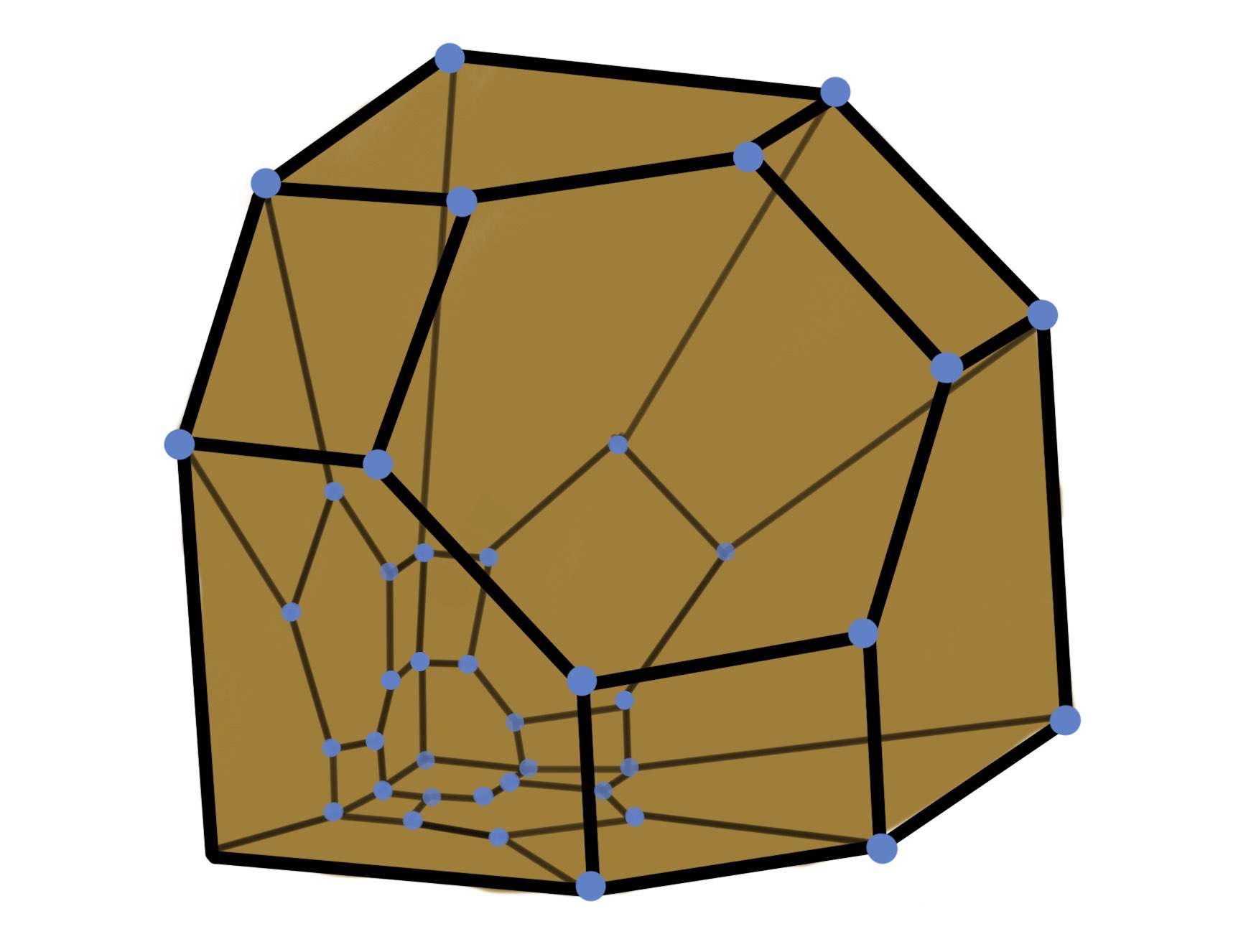}
        \caption{Schlegel diagram of $\mathcal{PF}_4$}
        \label{fig 7-a}
    \end{subfigure}
    \hfill
    \begin{subfigure}[b]{0.45\linewidth}
        \centering
        \includegraphics[width=\linewidth]{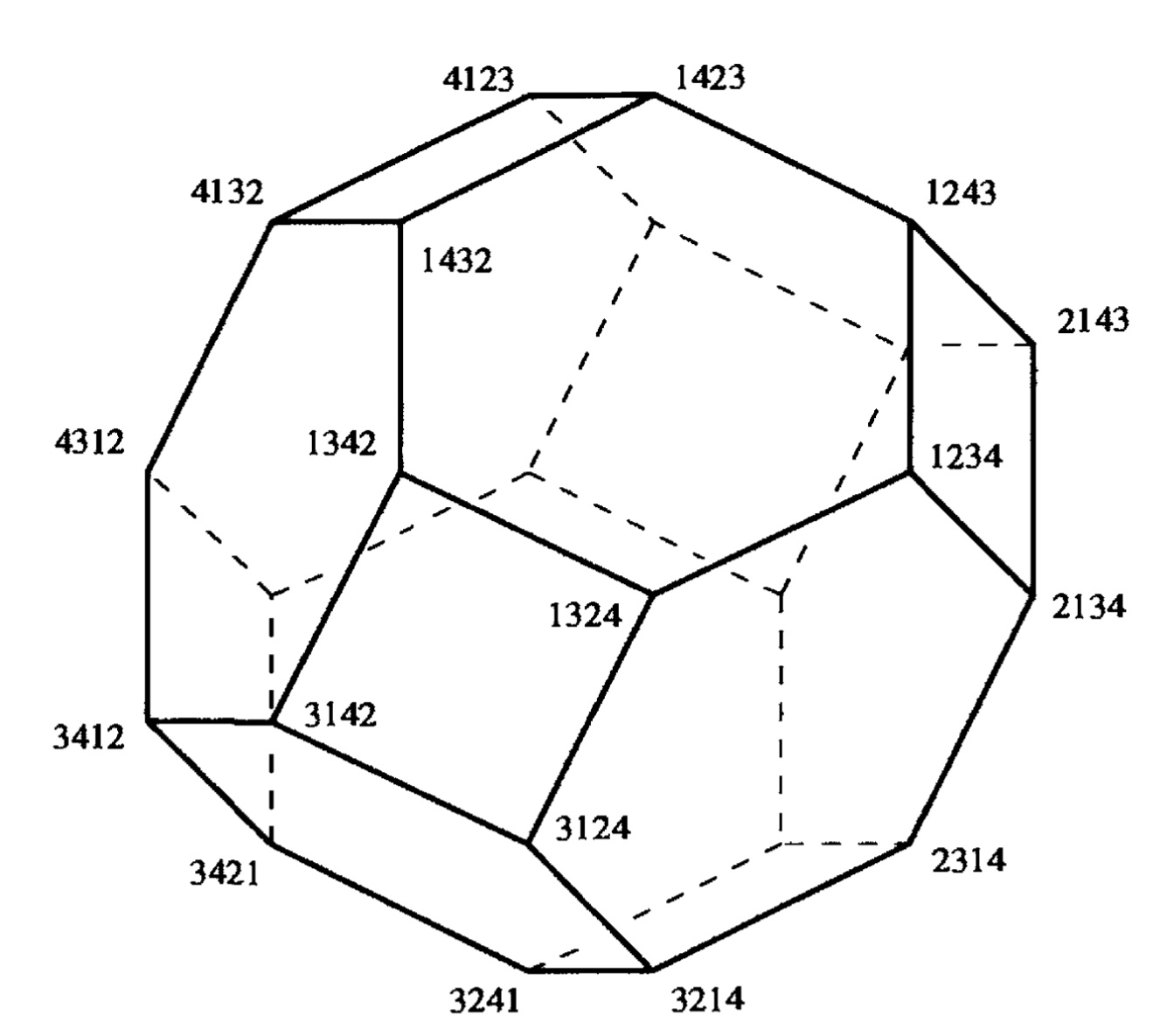}
        \caption{Permutahedron $\Pi_4$ \cite{ziegler1995lectures}}
        \label{fig 7-b}
    \end{subfigure}
    \caption{}
    \end{figure}
    
    If the reader is interested in a general overview about the face structure of $\mathcal{PF}_n$, see \cite[Section 2.1]{GeneralizedPFPolytope2023}.
\end{example}

\end{document}